\documentclass[a4paper,10pt,twopages,reqno,english]{article}

\author{Pietro Ploner}
\title{Computation of framed deformation functors}
\date{}

\usepackage{amsmath,amsfonts,amssymb}
\usepackage{amsthm}
\usepackage{makeidx}
\usepackage{graphicx}
\usepackage{amscd}
\usepackage[all,2cell]{xy} \SilentMatrices

\usepackage[latin1]{inputenc}

\newtheorem{defin}{Definition}[section]
\newtheorem{teo}[defin]{Theorem}
\newtheorem{lemma}[defin]{Lemma}
\newtheorem{prop}[defin]{Proposition}

\begin{document}

\maketitle

\begin{abstract}
In this work we compute the framed deformation functor associated to a reducible representation given as direct sum of 2-dimensional representations associated to elliptic curves with appropriate local conditions. Such conditions arise in the works of Schoof and correspond to reduction properties of modular elliptic curves.
\end{abstract}

\vskip 2cm

\section{Introduction}
\vskip 1.5cm

This work originates from the articles of Schoof about classification of abelian varieties \cite{S2}. There he examines the case of abelian varieties over $\mathbb{Q}$ with semistable reduction in only one prime $\ell$ and good reduction everywhere else and proves that they do not exist for $\ell=2,3,5,7,13$, while for $\ell=11$ they are given by products of the Jacobian variety $J_0(11)$ of the modular curve $X_0(11)$. In \cite{S3} he makes some generalisations of this result when $\ell$ is not a prime and the base field is not $\mathbb{Q}$, but a quadratic field. Some similar results, given in terms of $p$-divisible groups, were also previously obtained by Abrashkin in \cite{A}.

The main purpose of this work is to translate some results of those articles in terms of deformation theory of representations associated to elliptic curves. We examine the following setting: let $p\ne{\ell}$ be distinct primes and $S=\{p,\ell,\infty\}$. Let $\bar{\rho}_i:G_S\rightarrow{GL_2(k)}\ \ i=1,\dots,n$ be Galois representations, where $k$ is a finite field of characteristic $p$ and $G_S$ is the Galois group of the maximal extension of $\mathbb{Q}$ unramified outside $S$. We can suppose that there are exactly $r$ non-isomorphic representations among them and, up to reordering indexis, suppose that they are $\bar{\rho}_1,\dots,\bar{\rho}_r$. Then we can write
\begin{eqnarray}
\bar{\rho}=\bar{\rho}_1\oplus\dots\oplus{\bar{\rho}_n}=\oplus_{i=1}^r\bar{\rho}_i^{e_i}:G_S\rightarrow{GL_{2n}(k)}
\end{eqnarray}
where $\sum_{i=1}^r{e_i}=n$. The main result is the following

\begin{teo}: Suppose that:
\begin{enumerate}
	\item the $k$-vector space $Ext_{\underline{D},p}^1(V_{\bar{\rho}_i},V_{\bar{\rho}_j})$ of killed-by-$p$ extensions is trivial for every $i,j=1,\dots,r$;
	\item $Hom_G(V_{\bar{\rho}_i},V_{\bar{\rho}_j})=\begin{cases} k & \text{if}\  i=j \\ 0 & \text{if}\  i\ne{j} \end{cases}.$  
\end{enumerate}

Let $F_{\underline{D}}^{\Box}$ be the framed deformation functor associated to $\bar{\rho}$ with the local conditions:
\begin{itemize}
	\item $\rho$ is $p$-flat over $\mathbb{Z}[1/\ell]$;
	\item $\rho$ satisfies $(\rho_i(g)-Id)^2=0$ for every $g\in{I_{\ell}}$;
	\item $\rho$ is odd.
\end{itemize}

Then $F_{\underline{D}}^{\Box}$ is represented by a framed universal ring $R_{\underline{D}}^{\Box}$ which is isomorphic to $W(k)[[x_1,\dots,x_N]]$, where $N=4n^2-\sum_{i=1}^r{e_i}^2$.
\end{teo}
\vskip 1cm

The setting works in particular when $\bar{\rho_i}$ is the representation associated to the $p$-torsion points of an elliptic curve $E_i$ over $\mathbb{Q}$ having semistable reduction in $\ell$ and good supersingular reduction at $p$, as the varieties described in \cite{S2}. Moreover the local condition in $\ell$ corresponds to the condition of semistable action described in \cite[Section 2]{S2}, while the condition in $p$ is the classical flatness condition introduced in \cite{R}. The final result want to express that the framed deformation ring turns out to be the ``simplest" possible, giving an analog for deformation of \cite[Th. 8.3]{S2}.


The work is structured as such: in section 2 we recall the main features of deformation theory, following mainly Mazur's original formulation (see \cite{M1,M2}) and we introduce framed deformation functors, following Kisin's approach\cite{K1}, but avoiding the use of the formalism of groupoids. Then in the following three chapters we describe the local conditions we have used in our theorem: first we deal with the flatness condition in the residual prime $p$, which is treated according the initial results of Ramakrishna \cite{R} and expanded by the work of Conrad \cite{CDT,C1,C2}; then we pass to examine the semistability condition, which is treated as a particular case of representation of Steinberg type: the computation of the universal ring in this part is pretty indirect and passes through the use of formal schemes \cite{K2}. Finally, for the archimedean prime, the computation of the deformation ring is performed directly. Chapters 6 and 7 are dedicated to local-to-global arguments, which consent to build up a presentation of a global deformation ring using the computations on the local ones we did in the previous chapters. The definition of geometric deformation functor is introduced, too. It is due to Kisin \cite{K2} and the name comes form the ``geometric" representations defined in Serre's conjecture. In the final chapter Theorem 1.1 is proved with the use of the technical instruments introduced and through a final direct computation of the universal framed deformation, using matrix algebras.


\vskip 1cm

\section{Basics Notions of Deformation theory}
\vskip 1cm

Let $k$ be a finite field of characteristic $p$ and let $S$ be a finite set of rational primes containing $p$ and the archimedean prime. Denoting by $\mathbb{Q}_S$ the maximal extension of $\mathbb{Q}$ unramified outside $S$ and by $G_S=Gal(\mathbb{Q}_S/\mathbb{Q})$, consider a Galois representation $\bar{\rho}:G_S\rightarrow{GL_N(k)}$ and denote by $V_{\bar{\rho}}$ the associated $G_S$-module. Let $A$ be a complete noetherian local $W(k)$-algebra with residue field $k$ and denote by $\underline{\hat{Ar}}$ the category of such algebras. A {\it lift} of $\bar{\rho}$ to $A$ is a representation $\rho_A:G_S\rightarrow{GL_N(A)}$ such that the diagram

\[ \xymatrix {  & GL_N(A) \ar[d]^{{\pi}_A} \\ G_S \ar[ur]^{\rho_A} \ar[r]^{\bar{\rho}}   & GL_N(k) } \]
commutes, where ${\pi}_A$ is the natural projection on the residue field. Two lifts $\rho_1,\rho_2$ are said to be {\it equivalent} if there exists a matrix $M\in{Ker(\pi_A)}$ such that $M\rho_1(g)M^{-1}=\rho_2(g)$ for every $g\in{G_S}$.

A {\it deformation} of $\bar{\rho}$ to $A$ is an equivalence class of lifts. The starting representation $\bar{\rho}$ is called {\it residual}.

Deformations can be better understood via a categorical approach. Given a Galois representation $\bar{\rho}$, we can define the {\it deformation functor} 
\begin{eqnarray}
F_{\bar{\rho}}:\underline{\hat{Ar}}\rightarrow{\underline{Sets}},
\end{eqnarray} 
which associates to an element $A\in{\hat{Ar}}$ the set of deformation classes of $\bar{\rho}$ to $A$.

Kisin has introduced a variant of the deformation functor. Let $\beta$ be a $k$-basis of the Galois module $V_{\bar{\rho}}$. A {\it framed deformation} of the couple $(V_{\bar{\rho}},\beta)$ to a ring $A\in{\underline{\hat{Ar}}}$ is a couple $(V_A,\beta_A)$, where $V_A$ is a free $N$-dimensional $A$-module with continuous $G_S$-action lifting the action of $V_{\bar{\rho}}$ and $\beta_A$ is an $A$-basis of $V_A$ lifting $\beta$. We can then define the {\it framed deformation functor} $F_{\bar{\rho}}^{\Box}:\underline{\hat{Ar}}\rightarrow{\underline{Sets}}$ which associates to an algebra $A\in{\underline{\hat{Ar}}}$ the set of framed deformation classes of $(V_{\bar{\rho}},\beta)$ to $A$.

\begin{teo}
\begin{enumerate}
 \item The framed deformation functor is representable by a ring $R^{\Box}=R_{\bar{\rho}}^{\Box}\in{\underline{\hat{Ar}}}$. 
 \item If $\bar{\rho}$ satisfies the trivial centralizer condition $End_{k[G]}(V_{\bar{\rho}})\simeq{k}$, then $F_{\bar{\rho}}$ is representable by a ring $R=R_{\bar{\rho}}\in{\underline{\hat{Ar}}}$.
\end{enumerate}
\end{teo}

The rings $R$ and $R^{\Box}$ are called the {\it universal deformation ring} and the {\it universal framed deformation ring} of $\bar{\rho}$ respectively. They are universal in the sense that any deformation of $\bar{\rho}$ to an element $A\in{\underline{\hat{Ar}}}$ can be recovered via a unique homomorphism $R\rightarrow{A}$.

Among the algebras in $\underline{\hat{Ar}}$ there is one with particular properties. Let $k[\epsilon]$ be the ring of polynomials in $\epsilon$ with the condition $\epsilon^2=0$ and let $F$ be a deformation functor. The {\it tangent space} of $F$ is the set $F(k[\epsilon])$. It has a natural structure of $k$-vector space.

\begin{prop}
\begin{enumerate}
	\item $F_{\bar{\rho}}(k[\epsilon])$ is a finite dimensional $k$-vector space; 
	\item $F_{\bar{\rho}}(k[\epsilon])\simeq{H^1(G_S,Ad(\bar{\rho}))}\simeq{Ext_{k[G]}^1(V_{\bar{\rho}},V_{\bar{\rho}})}$ as $k$-vector spaces;
	\item $dim_{k}F_{\bar{\rho}}^{\Box}(k[\epsilon])=dim_{k}F_{\bar{\rho}}(k[\epsilon])+N^2-dim_{k}H^0(G_S,Ad(\bar{\rho})).$
\end{enumerate}
\end{prop}



Let $F_{\bar{\rho}}$ be a deformation functor and let $\underline{P}$ be the category of pairs $(A,V_A)$ with $A\in{\underline{\hat{Ar}}}$ and $V_A\in{F_{\bar{\rho}}(A)}$. Let $\underline{D}$ be a full subcategory of $\underline{P}$. We say that $\underline{D}$ is a {\it deformation condition} if the following conditions hold: 
\begin{enumerate}
	\item if $(A,V_A)\rightarrow{(B,V_B)}$ is a morphism in $\underline{P}$ and $(A,V_A)\in{\underline{D}}$, then $(B,V_B)\in{\underline{D}}$;
	\item if $(A,V_A)\rightarrow{(B,V_B)}$ is an injective morphism in $\underline{P}$ and $(B,V_B)\in{\underline{D}}$, then $(A,V_A)\in{\underline{D}}$;
	\item if $(A\times_{C}B,V)$ lies in ${\underline{D}}$, then also the projections $(A,V_A)$ and $(B,V_B)$ do. 
\end{enumerate}
 
Given a deformation condition $\underline{D}$, we can consider the functor $F_{\bar{\rho},\underline{D}}:\underline{\hat{Ar}}\rightarrow{\underline{Sets}}$ that sends a ring $A\in{\underline{\hat{Ar}}}$ to the set of deformations $\rho$ of $\bar{\rho}$ to $A$ such that $(A,V_{\rho})\in{\underline{D}}$. If $F_{\bar{\rho}}$ is representable by a ring $R_{\bar{\rho}}$, then $F_{\bar{\rho},\underline{D}}$ is representable, too, by a quotient of $R_{\bar{\rho}}$. Moreover the tangent space $F_{\bar{\rho},\underline{D}}(k[\epsilon])$ is a $k$-vector subspace of $F_{\bar{\rho}}(k[\epsilon])$.
\vskip 0.5cm

Given a Galois representation $\bar{\rho}:G_S\rightarrow{GL_N(k)}$ a natural way to attach a deformation condition is the following: for every $\ell\in{S}$ we consider the restriction $\bar{\rho}_{\ell}:G_{\ell}\rightarrow{GL_N(k)}$ together with a local deformation condition $\underline{D_{\ell}}$. Then we can define a global deformation condition $\underline{D}$ given by the objects $(A,V_A)\in{\underline{P}}$ whose local restriction to $\ell$ lies in $\underline{D_{\ell}}$ for every $\ell\in{S}$.
\vskip 0.5cm

Another important example of deformation condition is given by the fixed determinant. Let $\chi:G_S\rightarrow{W(k)^*}$ be a linear character. We say that a representation $\rho:G_S\rightarrow{GL_N(A)}$ {\it has determinant $\chi$} if the $G_S$-action induced on the wedge product $\Lambda^N(V_{\rho})$ is given by the character
\begin{eqnarray}
\chi_A:G_S\stackrel{\chi}\rightarrow{W(k)^*}\stackrel{\phi}\rightarrow{A^*},
\end{eqnarray}
where $\phi$ is the restriction of the $W(k)$-algebra structure morphism. The subcategory $\underline{D}$ of pairs $(A,V_{\rho})$ such that $\rho$ has determinant $\chi$ is a deformation condition and the corresponing deformation functor is denoted as $F_{\bar{\rho}}^{\chi}$. Moreover we have that
\begin{eqnarray}
F_{\bar{\rho}}^{\chi}(k[\epsilon])\simeq{H^1(G_S,Ad^0(\bar{\rho}))},
\end{eqnarray}
where $Ad^0$ denotes the vector space of matrices with trace zero provided by the adjoint $G_S$-action.

\vskip 0.5cm

\vskip 1.5cm

\section{The local flat deformation functor} 
\vskip 0.5cm

In this section we want to deal with a local deformation condition which refers to the prime $p$, characteristic of the finite base field $k$. This condition was mainly studied by Ramakrishna in \cite{R} and then generalised by Conrad in \cite{C1},\cite{C2} and Kisin in \cite{K1}. From now on, we will only deal with representations of degree 2.
\vskip 0.3cm

\begin{lemma}(Ramakrishna) Let $\underline{C}$ be a full subcategory of $\underline{Rep}_k(G)$ closed under passage to subobjects, direct products and quotients and let $\underline{D}$ be the full subcategory of $\underline{P}$ of pairs $(A,V_A)$ such that $V_A\in{\underline{C}}$. Then $\underline{D}$ is a deformation condition.
\end{lemma}
 
\begin{proof}
We need to prove that $\underline{D}$ satisfies the three properties of deformation conditions. Property 1 comes from the fact that, if $(A,V_A)\rightarrow{(B,V_B)}$ is a morphism in $\underline{D}$, then $V_A\otimes_{A}B\simeq{V_B}$ and the tensor product can be recovered via direct sums and quotients. Property 2 comes from the closure under subobjects. Property 3 comes from the fact that the fiber product and its projections can be constructed via direct sums and quotients, too.
\end{proof}

Let $F$ be a finite extension of $\mathbb{Q}_p$ and $\bar{\rho}:G_F\rightarrow{GL_2(k)}$ be a Galois representation. If $\rho$ is a deformation of ${\bar{\rho}}$ to a coefficient ring $A$, we say that $\rho$ is {\it flat} if there exists a finite flat group scheme $X$ over the ring of integers $O_F$ such that $V_{\rho}\simeq{X(\bar{F})}$, that is, $V_{\rho}$ is the generic fiber of $X$.
\vskip 0.2cm

\begin{prop} The subcategory $\underline{D}$ of flat deformations is a deformation condition.
\end{prop}
\vskip 0.1cm
\begin{proof} it suffices to show that $\underline{D}$ satisfies lemma 3.1. 

Let $0\rightarrow{T}\rightarrow{U}\rightarrow{V}\rightarrow{0}$ be a sequence of $G$-modules such that $U$ is the generic fiber of a finite flat group scheme $X$ over $O_F$. Then we can take the schematic closure $X_1$ of $T$ in $X$ (see \cite[Lemma 2.1]{R} for details) and $X_2=X/X_1$ to see that also $T$ and $V$ are generic fibers of finite flat group schemes. This argument and the fact that a direct sum of finite flat group schemes is still a finite flat group scheme show that the subcategory of flat deformations is a deformation condition. 
\end{proof}
\vskip 0.3cm

If $\bar{\rho}$ satisfies the trivial centralizer condition ${End}_{k[G_F]}(V_{\bar{\rho}})=k$, then the deformation functor which assigns to a coefficient ring the set of deformations of $\bar{\rho}$ which are flat, called the {\it flat deformation functor} and denoted as $F^{fl}$, is representable by a noetherian ring $R_p^{fl}$, which is called the {\it local flat universal deformation ring}. We want to give a proof of the main result of representability for this condition, which was proven by Ramakrishna for $p\ne{2}$ and by Conrad for all cases. First we need some technical data
\vskip 0.2cm

\begin{defin} Let $\phi$ denote the absolute Frobenius morphism. A {\it Fontaine-Lafaille module} is a $W(k)$-module $M$ provided with a decreasing, exhaustive, separated filtration of $W(k)$-submodules $\{M_i\}$ such that, for every index $i$, there exists a $\phi$-semilinear map $\phi_i:M_i\rightarrow{M}$ with the property that $\phi_i(x)=p\phi_{i+1}(x)$ for every $x\in{M}$. 
\end{defin}
\vskip 0.2cm

We denote by $MF$ the category of Fontaine-Lafaille modules over $W(k)$. Moreover we denote by $MF_{tor}^f$ the full subcategory of objects such that $M$ has finite length and $\sum{Im(\phi_i)}=M$ and by $MF_{tor}^{f,j}$ the subcategory of objects such that $M_0=M$ and $M_j=0$. Finally we say that a Fontaine-Lafaille module is {\it connected} if the morphism $\phi_0$ is nilpotent The main result about Fontaine-Lafaille modules is the following

\begin{teo}[Fontaine-Lafaille] For every $j\le{p}$ there exists a faithful exact contravariant functor
\begin{eqnarray} \label{Laffunc}
MF_{tor}^{f,j}\rightarrow{Rep_{\mathbb{Z}_p}^f(G)},
\end{eqnarray} 
which is fully faithful if $j<p$ and becomes fully faithful when restricted to the subcategory of connected Fontaine-Lafaille modules if $j=p$. Morevoer $MF_{tor}^{f,2}$ is antiequivalent to the category of finite flat group schemes over $W(k)$
\end{teo}

\begin{proof} See \cite[Ch.8-9]{FL} for a proof and description of the functor.
\end{proof}
\vskip 0.2cm

We say that a representation $\rho$ has weight $j$ if it comes from a Fontaine-Lafaille module lying in $MF_{tor}^{f,j}$ and we denote by $F_{j}$ the subfunctor of $F_{\bar{\rho}}$ given by deformations of $\bar{\rho}$ which are of weight $j$.
It follows that if $\bar{\rho}$ is flat, then the functors $F_2$ and $F_{fl}$ are the same, therefore we will identify them in the rest of the chapter.
\vskip 0.3cm
We can now prove the main result for flat deformation functor. The proof is due to Ramakrishna for the case $p>2$ (see \cite[section 3]{R}); then Conrad has shown (see \cite{CDT}) that the proof works also in the case $p=2$, since the Fontaine-Lafaille module used is connected.
\vskip 0.2cm

\begin{teo}(Ramakrishna) Let $\bar{\rho}:G_{\mathbb{Q}_p}\rightarrow{GL_2(k)}$ be a flat residual Galois representation with trivial centralizer and such that $det(\bar{\rho})=\chi$, where $\chi$ is the cyclotomic character. Then
\begin{eqnarray}
R_p^{fl}(\bar{\rho})\simeq{W(k)[[T_1,T_2]]}.
\end{eqnarray}
\end{teo}
\vskip 0.1cm
\begin{proof} We split the proof in two parts. Suppose first that $k=\mathbb{F}_p$ and $\bar{\rho}$ is the representation attached to the $p$-torsion points of an elliptic curve $E$ over $\mathbb{Q}_p$ with good supersingular reduction. We prove the theorem in this particular case, where computations are relatively easy, and then pass to the general case.

In the particular case we have chosen, we know that $\bar{\rho}$ satisfies the trivial centralizer hypothesis and is of weight 2. We calculate the tangent space $F_2(\mathbb{F}_p[\epsilon])$. Viewing $\mathbb{F}_p[\epsilon]^2$ as a 4-dimensional $\mathbb{F}_p$-vector space, we can see an element $\rho\in{F_2(\mathbb{F}_p[\epsilon])}$ as a matrix

\begin{eqnarray}
\rho(g)=\begin{pmatrix}\bar{\rho}(g) & 0 \\ R_{g} & \bar{\rho}(g)\end{pmatrix}
\end{eqnarray}
and such a representation gives clearly an element of $Ext_{2,p}^1(V_{\bar{\rho}},V_{\bar{\rho}})$, the extensions in the category of weight 2 representations which are killed by $p$. It is immediate to check that equivalence of litings correspond to equivalent extensions.

Let $M$ be the Fontaine-Lafaille module associated to $V_{\bar{\rho}}$ via Theorem 3.4. By full faithfulness of the functor, we have that $Ext_{2,p}^1(V_{\bar{\rho}},V_{\bar{\rho}})=Ext_{2,p}^1(M,M)$ and that $End_{MF}(M)=\mathbb{F}_p$.

We want to write the module in a compactified manner in terms of a $2\times{2}$-matrix $X_M$. For that we use the fact that $M_1$ is 1-dimensional (it will be proved shortly) and that $\phi_0(M_1)=0$. Then we write

\begin{eqnarray}
\phi_0=\begin{pmatrix} \alpha & 0 \\ \beta & 0\end{pmatrix},\ \ \ \phi_1=\begin{pmatrix} * & \gamma \\ * & \delta\end{pmatrix},\ \ X_M=\begin{pmatrix} \alpha & \gamma \\ \beta & \delta \end{pmatrix}.
\end{eqnarray}

The matrix $X_M$ encodes all the informations of the structure of $M$. We also want to write the elements of $Ext_{2,p}^1(M,M)$ via these matrices. If $N$ is such an element, we have

\begin{eqnarray}
X_N=\begin{pmatrix} X_M & C \\ 0 & X_M \end{pmatrix}, \ \ C\in{M_2(\mathbb{F}_p)}.
\end{eqnarray}

The matrix $C$ corresponds to an element of $Hom(M,M)$. If $N'$ is another element of $Ext_{2,p}^1(M,M)$ and $D$ is the $2\times{2}$ matrix in its upper triangular part, then it represents the same extension of $N$ if and only if there exist a matrix $\begin{pmatrix} Id & R \\ 0 & Id\end{pmatrix}\in{M_4(\mathbb{F}_p)}$ such that

\begin{eqnarray}
\begin{pmatrix}Id & R \\ 0 & Id\end{pmatrix}\begin{pmatrix}X_M & C \\ 0 & X_M\end{pmatrix}=\begin{pmatrix}X_M & D \\ 0 & X_M\end{pmatrix}\begin{pmatrix}Id & R \\ 0 & Id\end{pmatrix}
\end{eqnarray}
and this happens if and only if $C-D=[R,X_M]$. Moreover $R$ must preserve the filtration of $M$, because the isomorphism between $N$ and $N'$ does so. Let $\mathfrak{H}$ be the set of such matrices $R$. It follows that

\begin{eqnarray}
Ext_{2,p}^1(M,M)\simeq{Hom(M,M)/\{[R,X_M]:\ R\in{\mathfrak{H}}\}}
\end{eqnarray}

Now we know that $dim_{\mathbb{F}_p}M_0=2$ and $dim_{\mathbb{F}_p}M_2=0$. If $dim_{\mathbb{F}_p}M_1\ne{1}$ then any endomorphism of $M$ does not need to respect any filtration structure and therefore the centralizer of $X_M$ in $M_2(\mathbb{F}_p)$, which has at least dimension 2, would belong to $End_{MF}(M)$; this is impossible because the endomorphism ring is 1-dimensional. Therefore $dim_{\mathbb{F}_p}M_1={1}$.

Now we can compute the dimension of the tangent space: observe that $Hom(M,M)$ has dimension 4, the set of matrices $R$ which preserves the filtration of $M$ has dimension 3 and the kernel of the map $R\rightarrow{[R,X_M]}$ has dimension 1 (it is isomorphic to $End_{MF}(M))$. Therefore the tangent space has dimension $4-(3-1)=2$.

Now we have that $R_2(\bar{\rho})=\mathbb{Z}_p[[T_1,T_2]]/I$. We count the number of $\mathbb{Z}_p/p^{l}$-valued points of the universal ring, which is the number of objects $N\in{MF_{tor}^{f,2}}$ which are free $\mathbb{Z}_p/p^{l}$-modules of rank 2. If $N_p$ denotes the kernel of multiplication by $p$ in $N$, then we need $N_p\simeq{M}$, in terms of matrices, since  $X_N\equiv{X_M}\pmod{p}$. Since $X_N\in{M_2(\mathbb{Z}_p/p^l)}$ and we have to consider modulo $p$, there are $p^{4(l-1)}$ such matrices. We have to consider them modulo isomorphism. Now if $X_{N_1}\simeq{X_{N_2}}$, then there exists a matrix $R\in{M_2(\mathbb{Z}_p/p^l)}$ which respects the filtration of $M$ such that $RX_{N_1}=X_{N_2}R$; there are $p^{3(l-1)}$ such matrices and $p^{l-1}$ lie in the center of $M_2(\mathbb{Z}_p/p^l)$, therefore commute with all the $X_N$. So the number of $\mathbb{Z}_p/p^l$-valued points is $p^{4(l-1)}/(p^{3(l-1)}/p^{l-1})=p^{2(l-1)}$. Observe that this is the same number of $\mathbb{Z}_p/p^l$-valued points of $\mathbb{Z}_p[[T_1,T_2]]$.

Let now $f\in{I}$ and $(x,y)\in{(\mathbb{Z}_p/p^l)^2}$, then $f(x,y)\equiv{0}\pmod{p^l}$ for every positive integer $l$. It follows that, taking liftings to characteristic zero, $f(x,y)=0$ for all $(x,y)\in{(p\mathbb{Z}_p)^2}$ and therefore $f=0$. So $I=0$ and $R_2(\bar{\rho})=\mathbb{Z}_p[[T_1,T_2]]$.
\vskip 0.3cm

Now we can pass to the proof of the theorem in the general case and remove the hypothesis that $k=\mathbb{F}_p$ and that $\bar{\rho}$ is the representation coming from an elliptic curve. A lemma of Serre (whose proof can be found in \cite{R}) tells us that $\bar{\rho}$ has restriction to inertia given by

\begin{eqnarray}
\bar{\rho}|_I=\begin{pmatrix}\psi & 0 \\ 0 & \psi^p\end{pmatrix}
\end{eqnarray}
where $\psi$ is a fundamental character of level 2. For such a representation we will compute both the ``unrestricted" universal ring and the flat one. First of all we want to show that $H^2(G,Ad(\bar{\rho}))=0$. By Tate local duality we have that $H^2(G,Ad(\bar{\rho}))=H^0(G,Ad(\bar{\rho})^*)=(Ad(\bar{\rho})^*)^G$. Let $\phi\in{(Ad(\bar{\rho})^*)^G}$, we want to show that its kernel is 4-dimensional and therefore $\phi=0$. Let $R\in{Ad(\bar{\rho})}$, we have $\phi(gR)=g\phi(R)$, where the $G$-action is given by conjugacy composed with $\bar{\rho}$ on the left and by determinant on the right. It follows that, if $g\in{I}$, then $\bar{\rho}(g)R\bar{\rho}(g)^{-1}-det(\bar{\rho}(g))R\in{Ker(\phi)}$. Then, if we define the map
\begin{eqnarray}
T_g:R\rightarrow{\bar{\rho}(g)R\bar{\rho}(g)^{-1}-det(\bar{\rho}(g))R},
\end{eqnarray}
it suffices to show that there exists $g\in{I}$ such that $Ker(T_g)=0$. We choose a $g$ such that $\psi(g)=\alpha$ where $\alpha$ is an element of order $p^2-1$ in $k^*$. Then, taking explicit formulas

\begin{eqnarray}
R=\begin{pmatrix}x & y \\ z & w\end{pmatrix},\ \ T_g(R)=\begin{pmatrix}x(1-\alpha^{p+1}) & y(\alpha^{1-p}-\alpha^{p+1}) \\ z(\alpha^{p-1}-\alpha^{p+1}) & w(1-\alpha^{p+1})\end{pmatrix}
\end{eqnarray}
and the last matrix is zero if and only if $R=0$. Then our claim is proved.

Now we use the formula for the Euler-Poincar\'e characteristic for $Ad(\bar{\rho})$. Let $h^i=dim(H^i(G,Ad(\bar{\rho})))$. We have
\begin{eqnarray}
c_{EP}(Ad(\bar{\rho}))=h^0-h^1+h^2=-dim_{k}Ad(\bar{\rho}).
\end{eqnarray}
We have that $h^2=0$, $h^0=1$ (it is the trivial centralizer condition) and $dim_{k}Ad(\bar{\rho})=4$, therefore $h^1=5$. It follows that the unrestricted universal ring for such a representation is isomorphic to $W(k)[[T_1,T_2,T_3,T_4,T_5]]$.

The flat deformation ring can be computed by means of calculations similar to the ones performed in the case $k=\mathbb{F}_p$, except that we have to consider Fontaine-Lafaille modules over $W(k)$ instead of $\mathbb{Z}_p$ and all the dimensions have to be computed over $k$. We obtain again that $R_2(\bar{\rho})=W(k)[[T_1,T_2]]$. In particular $R_2(\bar{\rho})$ is a quotient of $R(\bar{\rho})$ and the surjective map between them has a $3$-dimensional kernel. The theorem is therefore proved.
\end{proof}

\vskip 0.3cm

Now we give a refinement of this result, which is due to Conrad \cite{C1}.
\vskip 0.2cm
\begin{teo} Let $\bar{\rho}$ be as in the previous theorem and let $F^{fl,\chi}$ be the subfunctor of flat deformations of $\bar{\rho}$ which have fixed determinant $\chi$. Then this functor is representable by the ring
\begin{eqnarray}
R_p^{fl,\chi}(\bar{\rho})\simeq{\mathbb{Z}_p[[T]]}.
\end{eqnarray}
\end{teo} \label{Conradteo}
\vskip 0.1cm

\begin{proof} For the proof see \cite[Ch. 4, Th.4.1.2]{C1}.
\end{proof}

\vskip 0.3cm

EXAMPLE: Let $E$ be an elliptic curve over $\mathbb{Q}_p$ that has supersingular reduction in $p$. Let $\bar{\rho}:G_{\mathbb{Q}_p}\rightarrow{GL_2(\mathbb{F}_p)}$ be the representation coming from the Galois action on the $p$-torsion points of $E$. Then, by the results of \cite{C2}, ${\bar{\rho}}$ is absolutely irreducible and therefore the functor $F_{\bar{\rho}}^{fl}$ is representable. Therefore, applying Ramakrishna's theorem, we have that the flat universal deformation ring is $\mathbb{Z}_p[[T_1,T_2]]$.
\vskip 1cm

\section{Steinberg representations at primes $\ell\ne{p}$}
\vskip 0.5cm

Now we want to analyse local conditions at finite primes which are different from $p$. We continue to assume that the representation space $V_{\bar{\rho}}$ has dimension 2.

\begin{defin}: A 2-dimensional representation $\bar{\rho}:G_{\ell}\rightarrow{GL_2(k)}$ is called of Steinberg type if it is a non-split extension of a character $\lambda:G_{\ell}\rightarrow{k^*}$ by the twist $\lambda(1)=\lambda\otimes{\chi_p}$ of $\lambda$ by the $p$-adic cyclotomic character $\chi_p$.
\end{defin} \label{Stein}
\vskip 0.2cm

A representation of Steinberg type has the matricial form
\begin{eqnarray}
\bar{\rho}(g)=\begin{pmatrix} \lambda(1)(g) & * \\ 0 & \lambda(g) \end{pmatrix}\ \ \ \forall{g}\in{I_{\ell}}
\end{eqnarray}
Observe that since $\ell\ne{p}$ the mod $p$ cyclotomic character is unramified and, if $p$ is not a square mod $\ell$, it also happens that $\chi_p$ and its twists are trivial. We do not impose any ramification restriction on the character $\lambda$. Up to twisting by the inverse character of $\lambda$, we may assume that $det(\bar{\rho})=\chi_p$ and that $V(\bar{\rho})(-1)^{G_{\ell}}\ne{0}$, which means that there is a subrepresentation of dimension 1 on which $G_{\ell}$ acts via $\chi_p$.
\vskip 0.3cm

We define a subfunctor 
\begin{eqnarray}
L_{\bar{\rho}}^{\chi_p}:\underline{\hat{Ar}}\rightarrow{\underline{Sets}}
\end{eqnarray}
of the deformation functor $F_{\bar{\rho}}^{\chi_p}$ as
\begin{eqnarray}
	L_{\bar{\rho}}^{\chi_p}(A)=(V_A,L_A)
\end{eqnarray}
where
\begin{itemize}
	\item $V_A$ is a deformation of $\bar{\rho}$ to $A$.
	\item $L_A$ is a submodule of rank 1 of $V_A$ on which $G_{\ell}$ acts via $\chi_p$.
\end{itemize}

We define in the same way the framed subfunctor $L_{\bar{\rho}}^{\chi_p,\Box}:\underline{\hat{Ar}}\rightarrow{\underline{Sets}}$ as

\begin{eqnarray}
	L_{\bar{\rho}}^{\chi_p,\Box}(A)=(V_A,\beta_A,L_A)
\end{eqnarray}
where
\begin{itemize}
	\item $(V_A,\beta_A)$ is a framed deformation of $\bar{\rho}$ to $A$.
	\item $L_A$ is a submodule of rank 1 of $V_A$ on which $G_{\ell}$ acts via $\chi_p$.
\end{itemize}

This is the subfunctor corresponding to liftings of Steinberg type. In the following we work with the framed setting to avoid representability problems.
\vskip 0.3cm

In order to deal with representability of deformations functors of Steingberg type, we need to recall the main definitions of formal schemes. Let $R$ be a noetherian ring and $I$ an ideal and assume that $R$ is $I$-adically complete, so that we have
\begin{eqnarray}
R=\lim_{\leftarrow}R/I^n.
\end{eqnarray} 
We define a topologycal space $Spf(R)$ in the following way: given an element $f\in{R}$ and $\bar{f}$ its reduction modulo $I$, we define $D(\bar{f})$ to be the set of prime ideals of $R/I$ not containing $\bar{f}$. Then the set $Spec(R/I)$ with the induced topology is called the {\it formal spectrum} of $R$, with respect to $I$, and denoted by $Spf(R)$. The sets $D(\bar{f})$ are a basis for the topology of $Spf(R)$.

\vskip 0.2cm

For each $f\in{R}$ we define
\begin{eqnarray}
R\langle{f^{-1}}\rangle=\lim_{\leftarrow}R[f^{-1}]/I^n
\end{eqnarray}
Then the assignment $D(\bar{f})\mapsto{R\langle{f^{-1}}\rangle}$ defines a structure sheaf on $Spf(R)$.
\vskip 0.2cm
\begin{defin} The affine formal scheme $Spf(R)$ over $R$ with respect to $I$ is the locally ringed space $(X,O_X)$, where $X=Spec(R/I)$ and $O_X(D(\bar{f}))=R\langle{f^{-1}}\rangle$ for each $f\in{R}$. 

A noetherian formal scheme is a locally ringed space $(X,O_X)$, where $X$ is a topological space and $O_X$ is a sheaf of rings over $X$ such that each point $x\in{X}$ admits a neighborhood $U$ such that $(U,O_X|_{U})$ is isomorphic to an affine formal scheme $Spf(R)$.
\end{defin}
\vskip 0.2cm

A morphism of formal schemes is a pair $(f,f^*):(X,O_X)\rightarrow{(Y,O_Y)}$, where $f:X\rightarrow{Y}$ is a continuous map of topologycal spaces and $f^*:O_Y\rightarrow{f_{*}O_X}$ is a morphism of sheaves.
\vskip 0.2cm

If $(X,O_X)$ is a scheme, we can obtain a formal scheme $\hat{X}$ by the following construction: let $I\subseteq{O_X}$ be an ideal sheaf and consider $\hat{X}$ the completion of $X$ along $I$. Its underlying topological space is given by the subscheme $Z$ of $X$ defined by $I$ and the structure sheaf is defined as before. A formal scheme obtained in this way is called {\it algebrizable}.
\vskip 0.2cm

Finally, given a functor $\underline{\hat{Ar}}\rightarrow{\underline{Sets}}$, we can pass to the opposite categories and obtain a functor $\underline{\hat{Ar}^{\circ}}\rightarrow{\underline{Sets^{\circ}}}$; $\underline{\hat{Ar}^{\circ}}$ is exactly the category of formal schemes on one point over $Spec(W(k))$ with residue field $Spec(k)$. Schlessinger's theorem then provides criteria for the functor to be representable by an object of $\underline{\hat{Ar}^{\circ}}$. 
\vskip 0.5cm

Let us now go back to deformation functors. We want to give a description of the representing object of $L_{\bar{\rho}}^{\chi_p,\Box}$ using formal schemes. Let $R=R_{\bar{\rho}}^{\chi_p,\Box}$ and $V_R$ be the 2-dimensional module over $R$ with the action given by the universal framed representation $\rho_{univ}^{\Box}$. Let $\mathbb{P}(R)$ be the projectivization of $V_R$ and $\hat{\mathbb{P}}(R)$ be its completion along the maximal ideal of $R$. We consider the closed subspace of $\hat{\mathbb{P}}(R)$ defined by the equations $gv-\chi(g)v=0$ for every $g\in{G_{\ell}}$ and each $v\in{\hat{\mathbb{P}}(R)}$. By formal GAGA, this subspace comes from a unique projective scheme $\mathfrak{L}$ over $Spf(R)$.

In the following we want to prove some properties of the scheme $\mathfrak{L}$. In particular we want to show that it is an affine scheme of the form $Spec(\tilde{R})$ for an appropriate ring $\tilde{R}$, which will be automathically the representing object of $L_{\bar{\rho}}^{\chi_p,\Box}$, because of the defining property of $\mathfrak{L}$.


\begin{lemma} $\mathfrak{L}$ is formally smooth over $Spec(W(k))$ and its generic fiber $\mathfrak{L}\otimes_{W(k)}W(k)[1/p]$ is connected.
\end{lemma}

\begin{proof} Let $A_1\rightarrow{A_2}$ be a sujective map in $\underline{\hat{Ar}}$. An element $\eta_2\in{L_{\bar{\rho}}^{\chi_p,\Box}(A_2)}$ corresponds to an extension $c(\eta_2)\in{Ext_{A_2[G]}^1(A_2,A_2(1))}$. If $\eta_1$ is a lift of $\eta_2$ to $L_{\bar{\rho}}^{\chi_p,\Box}(A_1)$, then the lift is uniquely determined by a lift of the class $c(\eta_2)$ to an element of $Ext_{A_1[G]}^1(A_1,A_1(1))$. Finding such a lift of extensions is equivalent to proving that the natural map
\begin{eqnarray}
H^1(G,\mathbb{Z}_p(1))\otimes_{\mathbb{Z}_p}M\rightarrow{H^1(G,M\otimes_{\mathbb{Z}_p}\mathbb{Z}_p(1))}
\end{eqnarray}
is an isomorphism for any $A\in{\underline{\hat{Ar}}}$ and any $A$-module $M$. Since $H^1$ commutes with direct sums, it is sufficient to prove this result for $M=\mathbb{Z}/p^n{\mathbb{Z}}$. In this case the map is trivially injective and its cokernel is given by $H^2(G,\mathbb{Z}_p(1))[p^n]$; but, by Tate's local duality, $H^2(G,\mathbb{Z}_p(1))$ is the Pontryagin dual of $\mathbb{Q}_p/\mathbb{Z}_p$, which has no $p^n$-torsion. Therefore the map is an isomorphism.

Now we need to prove connectedness. We denote by $\mathfrak{L}[1/p]$ the generic fiber. By smoothness, the schemes $\mathfrak{L}[1/p]$, $\mathfrak{L}$ and $\mathfrak{L}\otimes_{W(k)}k$ have the same number of connected components and, as schemes over $R$, the same is true for $\mathfrak{L}$ and $Z=\mathfrak{L}\otimes_R{k}$; by \cite[Prop.2.5.15]{K2} this scheme is either all of $\mathbb{P}(k)$, if the action of $G_{\ell}$ is trivial, or a single point. Therefore there is only one connected component.

\end{proof}
\vskip 0.3cm

Before going on, we need a further notation. Given $V$ a representation lifting $V_{\bar{\rho}_{\ell}}$ to some ring $A$, we denote by $F_{V}^{\chi}$ the subfunctor of $F_{\bar{\rho}_{\ell}}^{\chi}$ given by representations lifting $V$, too, and by $F_V^{\chi,\Box}$ the corresponding framed deformation functor.
\vskip 0.2cm

\begin{lemma} The natural morphism of functors $L_{V}^{\chi,\Box}\rightarrow{F_{V}^{\chi,\Box}}$ is fully faithful. In particular, if $V$ is indecomposable, the morphism is an equivalence, $F_{V}^{\chi}$ is representable and its tangent space is $0$-dimensional.
\end{lemma} \label{genfib}

\begin{proof} Let $A\in{\underline{\hat{Ar}}}$ and $(V_A,\beta_A)$ be a framed deformation and $L_A$ be a $\chi_p$-invariant line in $V_A$. We need to show that $L_A$ is unique. Indeed we have $Hom_{A[G]}(A(1),V_A/L_A)$ is trivial, since $det(V_A)=\chi$ and $V_A/L_A$ is free of rank 1. Therefore we have $Hom_{A[G]}(A(1),V_A)=Hom_{A[G]}(A(1),L_A)$ and the uniqueness follows.

Suppose now that $V$ is indecomposable, then in particular the unframed deformation functor $F_{V}^{\chi}$ is representable, too. We need to show that each deformation $V_A$ contains an $A$-line $L_A$ on which $G_{\ell}$ acts via $\chi$. For this, it is enough to show that the tangent space is $0$-dimensional, which implies that every deformation $V_A$ is isomorphic to ${V\otimes_k{A}}$ and therefore inherits the trivial $A$-line from $V_{\bar{\rho}}$. By Tate's local duality (see Section 3.1) we have
\begin{eqnarray}
h^1(G_{\ell},Ad^0(V))=h^0(G_{\ell},Ad^0(V))+h^2(G_{\ell},Ad^0(V)(1))
\end{eqnarray}
and that the two summands equal each other; therefore it is enough to show that $h^0(G_{\ell},Ad^0(V))=0$. Since $\ell\ne{2}$ we have the exact sequence
\begin{eqnarray}
0\rightarrow{H^0(G_{\ell},Ad^0(V))}\rightarrow{H^0(G_{\ell},Ad(V))}\rightarrow{H^0(G_{\ell},\mathbb{Q}_{\ell})}\rightarrow{0}.
\end{eqnarray}
Trivially $H^0(G_{\ell},\mathbb{Q}_{\ell})=\mathbb{Q}_{\ell}$ and $H^0(G_{\ell},Ad(V))=\mathbb{Q}_{\ell}$ because $V$ is indecomposable, therefore $H^0(G_{\ell},Ad^0(V))=0$ and the lemma is proved.

\end{proof}

\begin{teo} Let $Spec(R_{\bar{\rho}}^{\chi,1,\Box})$ be the image of the natural morphism $\mathfrak{L}\rightarrow{Spec(R_{\bar{\rho}}^{\chi,\Box})}$. Then $R_{\bar{\rho}}^{\chi,1,\Box}$ is a domain of dimension 4 and $R_{\bar{\rho}}^{\chi,1,\Box}[1/p]$ is formally smooth over $W(k)[1/p]$. Moreover, for every $A\in{\underline{Ar}}$, a morphism $R_{\bar{\rho}}^{\chi,\Box}\rightarrow{A}$ factors through $R_{\bar{\rho}}^{\chi,1,\Box}$ if and only if the corresponding 2-dimensional representation is of Steinberg type.
\end{teo} \label{Steinberg}

\begin{proof} The scheme $\mathfrak{L}$ is smooth over $W(k)$ and connected. The ring $R_{\bar{\rho}}^{\chi,1,\Box}$ is the ring of global section of $\mathfrak{L}$ over $R_{\bar{\rho}}^{\chi,\Box}$, hence it must be a domain. 

If we invert $p$, lemma \ref{genfib} tells us that the generic fiber $\mathfrak{L}[1/p]$ is a closed subscheme of $Spec(R_{\bar{\rho}}^{\chi,\Box}[1/p])$, then it must be isomorphic to $Spec(R_{\bar{\rho}}^{\chi,1,\Box}[1/p])$; this proves that $R_{\bar{\rho}}^{\chi,1,\Box}[1/p]$ is formally smooth over $W(k)[1/p]$.

We now calculate the dimension. Since $R_{\bar{\rho}}^{\chi,1,\Box}$ has no nontrivial $p$-torsion, it is sufficient to calculate it on the generic fiber and add 1. Let $V$ be an indecomposable point. By lemma 4.4, we have that $F_{V}^{\chi}$ is representable with tangent space of dimension $0$, therefore the framed functor $F_V^{\chi,\Box}$ has a tangent space of dimension 3. This proves the claim. 

Finally, to prove the last statement, we use again lemma 4.4. A morphism factors through $R_{\bar{\rho}}^{\chi,1,\Box}$ if and only if it lifts to a unique point of $\mathfrak{L}$, that is, if and only if the corresponding representation space $V$ has a 1-dimensional subrepresentation where $G$ acts thorugh $\chi$. The theorem is therefore proved.
\end{proof}

\section{Computations of odd deformation rings} \label{infinite}
\vskip 0.5cm

In this section we will deal with local conditions at the infinite places, computing explicitly the deformation ring. Let
\begin{eqnarray}
\bar{\rho}_{\infty}:Gal(\mathbb{C}/\mathbb{R})\rightarrow{GL_2(k)}
\end{eqnarray}
be a local representation at the infinite place with $det(\bar{\rho}_{\infty}(\gamma))=-1\in{k}$, where $\gamma$ is a complex conjugation. Then, up to conjugation, $\bar{\rho}_{\infty}$ must be of one of these forms

\begin{enumerate}
	\item $p>2$, $\bar{\rho}_{\infty}(\gamma)=\begin{pmatrix}1 & 0 \\ 0 & -1\end{pmatrix}$.
	\item $p=2$, $\bar{\rho}_{\infty}(\gamma)=\begin{pmatrix}1 & 1 \\ 0 & 1\end{pmatrix}$.
	\item $p=2$, $\bar{\rho}_{\infty}(\gamma)=\begin{pmatrix}1 & 0 \\ 0 & 1\end{pmatrix}$.
\end{enumerate}

A framed deformation of $\bar{\rho}_{\infty}$ is determined by the image of $\gamma$, which is a matrix whose square is the identity and whose characteristic polynomial is $x^2-1$. To compute the universal framed ring explicitely, we consider
\begin{eqnarray}
R=W(k)[a,b,c,d]/I.
\end{eqnarray}
where $I$ is an ideal encoding the condition on the characteristic polynomial. The completion of $R$ to the kernel of the map
\begin{eqnarray}
R\rightarrow{k},\ \ \ \begin{pmatrix}a & b \\ c & d\end{pmatrix}\rightarrow{\bar{\rho}_{\infty}(\gamma)}
\end{eqnarray}
is the universal framed deformation ring of $\bar{\rho}_{\infty}$.
\vskip 0.3cm

We do the computation explicitly for the case 3 above (the other two cases being similar). Let $M$ be a lifting of $\bar{\rho}_{\infty}(\gamma)$ to the universal ring, then, imposing the conditions $Tr(M)=0$ and $det(M)=1$, we have
\begin{multline}
R=W(k)[a,b,c,d]/((1+a)+(1+d),(1+a)(1+d)+1-bc)={}\\ =W(k)[a,b,c]/(-(1+a)^2+1-bc)={} \\=W(k)[a,b,c]/(-2a-a^2-bc),
\end{multline}
and so $R^{\Box}=\simeq{W(k)[[a,b,c]]/(2a+a^2+bc)}$. In particular, if we invert $p$, it is a regular ring of dimension 2 over $W(k)$. A similar computation gives the same result in the other two cases.
\vskip 1cm

\section{Local to global arguments}
\vskip 0.5cm

In this chapter we want to give a presentation of a global deformation ring in terms of local ones. The results we use are due to Kisin \cite{K1} and will contemplate both the framed and the unframed setting. In the applications, the framed setting is mostly used, to avoid representability problems in the local rings.
\vskip 0.3cm 

Let $S$ be a finite set of primes including $p$ and the infinite prime and $\Sigma$ a subset of $S$ containing $p$ and the infinite prime too; in many application we will have $\Sigma=S$ and let $\bar{\rho}:G_S\rightarrow{GL_2(k)}$ be a residual representation. For each $v\in{\Sigma}$ we denote by $G_v$ the absolute Galois group of the field $\mathbb{Q}_v$ of $v$-adic numbers and by $\bar{\rho}_v=\bar{\rho}|_{G_v}$. Assume that $\bar{\rho}$ as well as all of the $\bar{\rho}_v$ for $v\in{\Sigma}$ satisfy the trivial centralizer hypothesis. Let $F_{\bar{\rho}}$ and $F_{\bar{\rho}_v}$ be the deformation functors associated to $\bar{\rho}$ and $\bar{\rho}_v$ respectively; since all of the functors are representable, we denote by $R_v^{\chi}$ the local universal deformation ring of $\bar{\rho}_v$ with fixed determinant equal to $\chi$ and by $R_S^{\chi}$ the universal deformation ring of $\bar{\rho}$ with fixed determinant equal to $\chi$. Finally we put
\begin{eqnarray}
R_{\Sigma}^{\chi}=\hat{\mathop{\otimes}}_{v\in{\Sigma}}R_v^{\chi}.
\end{eqnarray}
\vskip 0.2cm

Let
\begin{eqnarray}
{\theta}_i:H^i(G_S,Ad^0(\bar{\rho}))\rightarrow{\prod_{v\in{\Sigma}}H^i(G_v,Ad^0(\bar{\rho}))}
\end{eqnarray}
be the restriction map in cohomology. Following \cite{K1}, we denote by $r_i$ and $t_i$ the dimensions of the kernel and cokernel of ${\theta}_i$ as $k$-vector spaces.
\vskip 0.2cm
Let $m_{\Sigma}^{\chi}$ and $m_{S}^{\chi}$ be the maximal ideals of $R_{\Sigma}^{\chi}$ and $R_{S}^{\chi}$ respectively and let
\begin{eqnarray}
\eta:m_{\Sigma}^{\chi}/((m_{\Sigma}^{\chi})^2,p)\rightarrow{m_S^{\chi}/((m_S^{\chi})^2,p)}
\end{eqnarray}
be the map between the dual tangent spaces. Then we have the following result

\vskip 0.2cm

\begin{teo} If the functors $F_{\bar{\rho}}$ and $F_{\bar{\rho}_v}$ are representable, then there exist elements $f_1,\dots,f_{t_1+r_2}$ lying in the maximal ideal of $R_{\Sigma}^{\chi}[[x_1,\dots,x_{r_1}]]$ such that
\begin{eqnarray}
R_S^{\chi}=R_{\Sigma}^{\chi}[[x_1,\dots,x_{r_1}]]/(f_1,\dots,f_{t_1+r_2}).
\end{eqnarray}
In particular $dim_{Krull}R_S^{\chi}\ge{1+r_1-r_2-t_1}$
\end{teo} \label{keylemma}
\vskip 0.1cm
\begin{proof} Consider the quotient ring $R_S^{\chi}/m_{\Sigma}^{\chi}$; the tangent space of this ring is clearly the dual of $ker({\theta}_1)$, therefore these two vector spaces have the same dimension. This proves the claim on the number of variables.

Let now $I=ker(\eta)$. There exists a surjection $R_{gl}=R_{\Sigma}^{\chi}[[x_1,...,x_{r_1}]]\rightarrow{R_S^{\chi}}$ which induces a surjection on tangent spaces with kernel isomorphic to $I$. Denote by $m_{gl}$ the maximal ideal of $R_{gl}$ and by $J$ the kernel of the surjection. Let $\rho_S^{\chi}$ be the universal deformation of $\bar{\rho}$ and consider a set theoretic lift $\rho_{gl}$ of $\rho_S^{\chi}$ to the ring $R_{gl}/Jm_{gl}$ with determinant $\chi$. Define now a 2-cocycle
\begin{eqnarray}
c:H^2(G_S,J/m_{gl}J\otimes_k{Ad^0(\bar{\rho})}),\  \ c(g_1,g_2)=\rho_{gl}(g_1g_2)\rho_{gl}(g_2)^{-1}\rho_{gl}(g_1)^{-1},
\end{eqnarray}
where we identify $J/m_{gl}J\otimes_k{Ad^0(\bar{\rho})}$ with the kernel of the natural projection map $GL_2(R_{gl}/m_{gl}J)\rightarrow{GL_2(R_{gl}/J)}$. It is easy to see that the class of $c$ in $H^2(G_S,Ad^0(\bar{\rho}))\otimes_k{J/m_{gl}J}$ does not depend on $\rho_{gl}$, but only on the universal deformation $\rho_S^{\chi}$ and is trivial if and only if $\rho_{gl}$ is a homomorphism.


Now, if we consider the restriction of $c$ to $H^2(G_p,Ad^0(\bar{\rho}))$, this is the trivial cocycle, because $\rho_S^{\chi}|_{G_p}$ has a natural lifting to $GL_2(R_{gl})$. Then $c\in{Ker(\theta_2)}\otimes_k{J/m_{gl}J}$. Let $(J/m_{gl}J)^*$ denote the $k$-dual, then we obtain a map
\begin{eqnarray}
\gamma:(J/m_{gl}J)^*\rightarrow{Ker(\theta_2)},\ \ \ \ \gamma(u)=\langle{c},u\rangle;
\end{eqnarray}
clearly $I^*\subseteq{(J/\tilde{m}J)^*}$, we claim that $Ker(\gamma)\subseteq{I^*}$.

Let $u\in{Ker(\gamma)}$ be a nonzero element; we denote by $R_{gl}^u$ the push-out of $R_{gl}/m_{gl}J$ by $u$, so that $R_S^{\chi}\simeq{R_{gl}^u/I^u}$, with $I^u$ an ideal of square zero and isomorphic to $k$ as an $R_{gl}^u$-module. Since $u\in{Ker(\gamma)}$ we can find a representation $\rho_u:G_{\mathbb{Q},S}\rightarrow{GL_2(\tilde{R}_u)}$ with determinant $\chi$ which lifts $\rho_S^{\chi}$. Then, by the universal property of $R_S^{\chi}$ the natural map $R_{gl}^u\rightarrow{R_S^{\chi}}$ has a section; it follows that $R_{gl}^u\simeq{R_S^{\chi}}\oplus{I^u}$ and $R_{gl}^u/pR_{gl}^u\simeq{R_S^{\chi}/pR_S^{\chi}\oplus{I_u}}$. Therefore the map $R_{gl}^u\rightarrow{R_S^{\chi}}$ does not reduce to an isomorphism on tangent spaces and it follows that the induced map
\begin{eqnarray}
Ker(J/m_{gl}J\rightarrow{I})\rightarrow{J/m_{gl}J}\rightarrow{I^u}
\end{eqnarray}
is not surjective and must be the zero map, that is, $u$ factors through $I$ and we have proved the claim.

Hence we have proved that 
\begin{eqnarray}
dim(J/m_{gl}J)=dim_{k}Ker(\gamma)+dim_{k}Im(\gamma)\le{dim(I)+r_2}=t_1+r_2
\end{eqnarray}
and we are done.
\end{proof}


\vskip 0.3cm
The hypotheses of the theorem are too strong for concrete applications, because they require all the functors to be representable. Therefore we want to establish a similar result in the framed setting. For the rings and ideals we have already defined, we simply add the $\Box$ superscript to indicate that we are in the framed case
\vskip 0.2cm

We need to define an auxiliary functor 
\begin{eqnarray}
F_{\Sigma,S}^{\chi,\Box}:\underline{\hat{Ar}}\rightarrow{\underline{Sets}}
\end{eqnarray} 
which associates to every coefficient ring $A$ a deformation of $\bar{\rho}$ to $A$ and a $\Sigma$-tuple of bases of $V_A$ in the following way:

\begin{eqnarray}
F_{\Sigma,S}^{\chi,\Box}(A)=\{(V_A,\iota_A,(\beta_v)_{v\in{\Sigma}})|(V_A,\iota_A)\in{F_{\bar{\rho}}^{\chi}(A)},\iota_A(\beta_v)=\beta\ \forall{v\in{\Sigma}}\}/\simeq.
\end{eqnarray}
\vskip 0.2cm

We have natural morphisms of functors 
\[ \xymatrix { F_{\Sigma,S}^{\chi,\Box} \ar[r] \ar[d] & \prod_{v\in{\Sigma}}F_{\rho_v} \\
               F_{\bar{\rho}}^{\chi} & }, \]
where the orizontal map is the restriction modulo each $v\in{\Sigma}$ and the vertical map is simply the forgetful functor which ignores bases. The following proposition describes the nature of these morphisms.
\vskip 0.2cm

\begin{prop} The natural morphism $F_{\Sigma,S}^{\chi,\Box}\rightarrow{F_{\bar{\rho}}^{\chi}}$ is smooth and, passing to universal ring, we have an isomorphism
\begin{eqnarray}
R_{\Sigma,S}^{\chi,\Box}\simeq{R_S^{\chi}[[x_1,\dots,x_{4|\Sigma|-1}]]}.
\end{eqnarray}
Moreover the morphism $F_{\Sigma,S}^{\chi,\Box}\rightarrow{\prod_{v\in{\Sigma}}F_{\rho_v}}$ gives a homomorphism of universal rings
\begin{eqnarray}
R_{loc}=\hat{\otimes}_{v\in{\Sigma}}R_v^{\chi,\Box}\rightarrow{R_{\Sigma,S}^{\chi,\Box}}.
\end{eqnarray}
\end{prop}

\begin{proof} The smoothness and the dimension formula come from the smoothness of the framed deformation functor over the unframed one and the morphism of universal rings comes naturally from the morphism of functors.
\end{proof}

\vskip 0.2cm

The passage to local rings is the key for computing $R_{\Sigma,S}^{\chi,\Box}$. The use of framed deformations avoids the representability issues.
\vskip 0.2cm

We can now state one of the main results of this approach. We need a generalization of the map $\theta_1$, defined at the beginning of the chapter
\vskip 0.2cm

\begin{lemma}[Key lemma] Let 
\begin{eqnarray}
\theta_1^{\Box}:F_{\Sigma,S}^{\chi,\Box}(k[\epsilon])\rightarrow{\mathop{\oplus}_{v\in{\Sigma}}F_{\rho_v}^{\chi,\Box}(k[\epsilon])}
\end{eqnarray}
be the restriction map on tangent spaces and set $r=dim_{k}Ker(\theta_1^{\Box})$ and $t=dim_{k}Ker(\theta_2)+dim_{k}coKer(\theta_1^{\Box})$. Then we have a presentation
\begin{eqnarray}
R_{\Sigma,S}^{\chi,\Box}\simeq{R_{loc}[[x_1,...,x_r]]/(f_1,...,f_t)}
\end{eqnarray}
\end{lemma}

\begin{proof} The proof is the same of theorem 6.1 in the unframed setting, simply substituting $R_S^{\chi}$ with $R_{\Sigma,S}^{\chi,\Box}$ and the cohomological groups and the map $\theta$ with their framed counterparts.
\vskip 0.2cm
\end{proof}

Observe that $r$ is an optimal value, while $t$ is just an upper bound on the number of relations; for example some of the $f_i$ may be trivial.
\vskip 0.3cm

Now we need also to evaluate $\delta=dim_{k}coKer(\theta_2)$. Note that $\theta_2$ is part of the Poitou-Tate sequence (see \cite{W} for references) and that $H^2(G_v,Ad^0(\bar{\rho}))\simeq{H^0(G_v,Ad^0(\bar{\rho})^*)^*}$ by local Tate duality. Therefore
\begin{multline}
\delta=dim_{k}coKer(\theta_2)={} \\
=dim_{k}Ker(H^0(G_S,Ad^0(\bar{\rho})^*)\rightarrow{\mathop{\oplus}_{v\in{S\backslash{\Sigma}}}H^0(G_v,Ad^0(\bar{\rho})^*)}).
\end{multline}
Note that $\delta=0$ if $S\setminus{\Sigma}$ is non-empty, and therefore contains a finite prime, or if the image of $\bar{\rho}$ is non-solvable, and therefore $H^0(G_S,Ad^0(\bar{\rho})^*)$ is trivial.
\vskip 0.3cm

The following result gives us a link between all the quantities we have defined
\vskip 0.2cm

\begin{teo} If $\Sigma$ contains all the places above $p$ and $\infty$, then $r-t+\delta=|\Sigma|-1$.
\end{teo}
\vskip 0.1cm
\begin{proof} We will make use of the Tate's computation of the Euler-Poincar\'e characteristic (a proof of which can be found in \cite{W})
\begin{eqnarray}
c_{EP}(G_S,Ad^0(\bar{\rho}))=-dim_{k}Ad^0(\bar{\rho})+h^0(G_{\infty},Ad^0(\bar{\rho}_{\infty}))
\end{eqnarray}
and the local version

\begin{eqnarray}
c_{EP}(G_v,Ad^0(\bar{\rho}_v))=\begin{cases} -dim_{k}Ad^0(\bar{\rho}_v) & \text{if}\  v=p \\ h^0(G_v,Ad^0(\bar{\rho}_v)) & \text{if}\  v=\infty \\ 0 & \text{otherwise}\end{cases}
\end{eqnarray}

Then we have

\begin{multline}
r-t+\delta=dim_{k}Ker(\theta_1^{\Box})-dim_{k}coKer(\theta_1^{\Box})-dim_{k}Ker(\theta_2)+{} \\ +dim_{k}coKer(\theta_2)=dim_{k}F_{\Sigma,S}^{\chi,\Box}(k[\epsilon])-\sum_{v\in{\Sigma}}dim_{k}F_v(k[\epsilon])-h^2(G_S,Ad^0(\bar{\rho}))+{} \\ +\sum_{v\in{\Sigma}}h^2(G_v,Ad^0(\bar{\rho}_v)).
\end{multline}

Now we evaluate the dimensions of tangent spaces and we have

\begin{multline}
h^1(G_S,Ad^0(\bar{\rho}))-h^0(G_S,Ad^0(\bar{\rho}))-1+|\Sigma|n^2-h^2(G_S,Ad^0(\bar{\rho}))+{} \\ -\sum_{v\in{\Sigma}}(h^1(G_v,Ad^0(\bar{\rho}_v))-h^0(G_v,Ad^0(\bar{\rho}_v))-1+n^2-h^2(G_v,Ad^0(\bar{\rho}_v)))={} \\ =-c_{EP}(G_S,Ad^0(\bar{\rho}))+\sum_{v\in{\Sigma}}c_{EP}(G_v,Ad^0(\bar{\rho}_v))+|\Sigma|-1.
\end{multline}

Finally we use Tate's formulas for $c_{EP}$ and we have

\begin{multline}
dim_{k}Ad^0(\bar{\rho})-h^0(G_{\infty},Ad^0(\bar{\rho}_{\infty}))-dim_{k}Ad^0(\bar{\rho}_p)+{} \\ +h^0(G_{\infty},Ad^0(\bar{\rho}_{\infty}))+|\Sigma|-1=|\Sigma|-1.
\end{multline}
\end{proof}

\vskip 1cm

\section{Geometric deformation rings}
\vskip 0.5cm

In this chapter we want to give some results about a particular class of deformation problems. We suppose that our ${\bar{\rho}}$ is odd of dimension 2 and absolutely irreducible (so that the deformation functor is representable). The rest of the notation matches the one of the previous chapter.
\vskip 0.3cm

For each $v\in{\Sigma}$ let $\tilde{F}_v^{\chi,\Box}$ be a representable subfunctor of ${F_v^{\chi,\Box}}$ such that the corresponding representing ring $\tilde{R}_v^{\chi,\Box}$ (which is a quotient of $R_v^{\chi,\Box}$) is flat over $\mathbb{Z}_p$ and satisfies:
\begin{eqnarray}  
	 dim_{Krull}\tilde{R}_v^{\chi,\Box}[1/p]= \begin{cases} 3 & \text{if}\  v\ne{p,\infty} \\ 4 & \text{if}\  v=p \\ 2 & \text{if}\  v=\infty \end{cases}  
\end{eqnarray} 
\vskip 0.3cm

A deformation functor satisfying these properties will be called a {\it geometric deformation functor}. The name was given by Kisin to match the results on {\it geometric} representations defined in Fontaine-Mazur's Conjecture (See \cite{K2} for details) which all share this property. Such a functor satisfies the following properties:
\begin{itemize}
	\item $\tilde{R}_{loc}=\hat{\mathop{\otimes}}_{v\in{\Sigma}}\tilde{R}_v^{\chi,\Box}$ is flat over $\mathbb{Z}_p$ and its Krull dimension is $\ge{3|\Sigma|+1}$.
	\item The functors $\tilde{F}_S^{\chi}$ and $\tilde{F}_{\Sigma,S}^{\chi,\Box}$ are representable.
	\item The ring $\tilde{R}_{\Sigma,S}^{\chi,\Box}$ is isomorphic to $R_{\Sigma,S}^{\chi,\Box}\hat{\mathop{\otimes}}_{v\in{\Sigma}}\tilde{R}_{loc}$ and therefore 
	\begin{eqnarray}
	\tilde{R}_{\Sigma,S}^{\chi,\Box}\simeq{\tilde{R}_{loc}[[x_1,...,x_r]]/(f_1,...,f_t)}
	\end{eqnarray}
	with $r,t$ defined as before. In particular the Krull dimension of $\tilde{R}_{\Sigma,S}^{\chi,\Box}\ge{4|\Sigma|-\delta}$.
\end{itemize}
\vskip 0.3cm

Since the map $\tilde{F}_{\Sigma,S}^{\chi,\Box}\rightarrow{\tilde{F}_S^{\chi}}$ is smooth, we can obtain the following result
\vskip 0.2cm

\begin{teo} If $\delta=0$, then $dim_{Krull}\tilde{R}_S^{\chi}\ge{1}$.
\end{teo}

\vskip 1cm

\section{The main result}
\vskip 0.5cm

Now we have all the necessary instruments to generalize the results of the previous chapter. Let $\bar{\rho}_1,\dots,\bar{\rho}_n$ be representations of $G_{\mathbb{Q}}$ each with values in $GL_2(k)$, where $k$ is a finite field of characteristic $p$. Let $\underline{Gr}$ be the category of finite flat group scheme over $\mathbb{Z}_p$ of order a power of $p$ and let $\underline{D}$ be a subcategory of $\underline{Gr}$ closed by products, subobjects and quotients. We assume that each $V_{\bar{\rho}_i}$ is the generic fiber of an element of $\underline{D}$ . We write
\begin{eqnarray}
\bar{\rho}=\bar{\rho}_1\oplus{\dots}\oplus{\bar{\rho}_n}:G_{\mathbb{Q}}\rightarrow{GL_{2n}(k)}.
\end{eqnarray}
It may happen that some of the $\bar{\rho}_i$ are isomorphic. Therefore we suppose that there are exactly $r$ different representations among the $\bar{\rho}_i$ which are non-isomorphic and we assume them to be $\bar{\rho}_1,\dots,\bar{\rho_r}$. Then we rewrite $\bar{\rho}$ as
\begin{eqnarray}
\bar{\rho}=\mathop{\oplus}_{i=1}^r{\bar{\rho}_i^{e_i}}.
\end{eqnarray}
\vskip 0.2cm

We want to define a deformation functor in this case. We start considering the single representation $\bar{\rho}_i$. We define the deformation functor $F_{\bar{\rho}_i,\underline{D}}:\underline{\hat{Ar}}\rightarrow{\underline{Sets}}$ which sends an artinian ring $A$ to the set of deformation classes $\rho_i$ of $\bar{\rho_i}$ to $A$ such that
\begin{itemize}
	\item $\rho_i$ is $p$-flat over $\mathbb{Z}[1/\ell]$;
	\item $\rho_i$ satisfies $(\rho_i(g)-Id)^2=0$ for every $g\in{I_{\ell}}$;
	\item $\rho_i$ is odd;
\end{itemize}
and let $F_{\bar{\rho},\underline{D}}:\underline{Ar}\rightarrow{\underline{Sets}}$ be the deformation functor associated to $\bar{\rho}$ with the same local conditions.

\vskip 0.2cm


\begin{lemma} $F_{\bar{\rho}_i,\underline{D}}$ is a geometric deformation functor.
\end{lemma}


\begin{proof}
We need to show that our local conditions satisfy the definition of geometric functor defined in section 7. At the prime $p$ we apply theorem 3.6 which tells us that the local ring is isomorphic to $\mathbb{Z}_p[[X]]$; in particoular, after inverting $p$ its framed counterpart has Krull dimension 4, as in the geometric conditions. At the infinite prime, the computations of section 5 tells us that the dimension over $\mathbb{Z}_p$ of the framed deformation ring is 2. Finally at the prime $\ell$ the condition that $(\rho_i(\sigma)-id)^2=0$ is equivalent to a Steinberg type condition with $\lambda$ equal to the trivial character. Therefore theorem 4.5 gives us that the framed deformation ring has Krull dimension 4; in particular, inverting $p$ it is regular of dimension 3. It follows that all the conditions of being a geometric deformation functor are satisfied. Then we can apply theorem 7.1 and obtain that each $F_{\bar{\rho}_i}$ has a representing ring of Krull dimension at least 1.
\end{proof}

\vskip 0.2cm

\begin{teo}[Main theorem: dimension 2 case] Suppose that:
\begin{enumerate}
	\item $Ext_{\underline{D},p}^1(V_{\bar{\rho}_i},V_{\bar{\rho}_j})$ of killed-by-$p$ extensions is trivial for every $i,j=1,\dots,r$;
	\item $Hom_G(V_{\bar{\rho}_i},V_{\bar{\rho}_j})=\begin{cases} k & \text{if}\  i=j \\ 0 & \text{if}\  i\ne{j} \end{cases}.$  
\end{enumerate}

Then the functor $F_{\bar{\rho},\underline{D}}^{\Box}$ is represented by a power series ring over $W(k)$ in $N$ variables, where
\begin{eqnarray}
N=4n^2-\sum_{i=1}^r{e_i^2}.
\end{eqnarray}
\end{teo}

\begin{proof} The representation $\bar{\rho}$ has the following matrix form
\begin{eqnarray}
\begin{pmatrix}
\begin{pmatrix} 
\bar{\rho}_1 & & \\
 & \ddots &  \\
 &  & \bar{\rho}_1
 \end{pmatrix} & & \\
  & \ddots & \\
  & & \begin{pmatrix} 
\bar{\rho}_r & & \\
 & \ddots &  \\
 &  & \bar{\rho}_r
 \end{pmatrix}
 \end{pmatrix}.
\end{eqnarray}
We call $\bar{T}$ this matrix and $\bar{\beta}$ a $k$-basis of $V_{\bar{\rho}}$ in which $\bar{\rho}$ has this matrix form; $\bar{T}$ belongs to $M_h(k)$ where we denote by $h=2n$. We also denote by $h_j=\sum_{i=1}^{j-1}2e_i$.

By the lemma, we know that each $\bar{\rho}_i$ is geometric, therefore the funtor $F_{\bar{\rho}_i}$ is represented by a ring of Krull dimension $\ge{1}$; on the other hand the hypothesis of triviality of extension set tells us that the tangent space of $F_{\bar{\rho_i}}$ is trivial, therefore the universal deformation ring is a quotient of $W(k)$. Therefore the universal ring must be isomorphic to $W(k)$ and there exist a $p$-adic lift of $\bar{\rho}_i$, given by the universal representation, that we call $\rho_i$.

Let then $T$ be the matrix obtained by $\bar{T}$ replacing all the $\bar{\rho}_i$ with the respective $\rho_i$, $V_{\rho}$ the associated representation module over $W(k)$ and $\beta$ a basis of $V_{\rho}$ lifting $\bar{\beta}$ in which $T$ has the block-diagonal shape. We look for a framed deformation of $\bar{T}$ of the form
\begin{eqnarray}
\tilde{T}=(1+M(\underline{x}))T(1+M(\underline{x}))^{-1}
\end{eqnarray}
where $M=M(\underline{x})$ is the matrix having a variable $x_{i,j}$ as $(i,j)$-th entry and $\underline{x}$ is the array of all such $x$. We write $M$ as
\begin{eqnarray}
\begin{pmatrix}
M_{1,1} & M_{1,2} & \dots & M_{1,r} \\
M_{2,1} & M_{2,2} & & \vdots \\
\vdots & & \ddots & \\
M_{r,1} & M_{r,2} & \dots & M_{r,r}
\end{pmatrix}.
\end{eqnarray}
where $M_{i,j}$ is the $2e_i\times{2e_j}$ submatrix given by
\begin{eqnarray}
M_{i,j}=\begin{pmatrix} x_{h_i+1,h_j+1} & \dots & x_{h_i+1,h_{j+1}} \\
\vdots & \ddots & \vdots \\
x_{h_{i+1},h_j+1} & \dots & x_{h_{i+1},h_{j+1}} \end{pmatrix}.
\end{eqnarray}
Then we have that $(\tilde{T},\beta(1+M))$ gives a framed deformation of $\bar{\rho}$ to the ring $R=W(k)[[x_{1,1},\dots,x_{2n,2n}]]$. 

We want to modify our deformation by a linear transformations lying in the centralizer of $\rho$ to kill some of the variables. Consider the diagonal submatrices $M_{i,i}$; we can eventually subdivide it in $2\times{2}$ submatrices
\begin{eqnarray}
\begin{pmatrix} M_{i,i}^{(1,1)} & \dots & M_{i,i}^{(1,e_i)} \\
\vdots & \ddots & \vdots \\
M_{i,i}^{(e_i,1)} & \dots & M_{i,i}^{(e_i,e_i)}\end{pmatrix},
\end{eqnarray}
where
\begin{eqnarray}
M_{i,i}^{(s,t)}=\begin{pmatrix} x_{h_i+2s-1,h_i+2t-1} & x_{h_i+2s-1,h_i+2t} \\
x_{h_i+2s,h_i+2t-1} & x_{h_i+2s,h_i+2t} \end{pmatrix}.
\end{eqnarray}

We look for a matrix $Y\in{M_h(R)}$ such that $1+Y$ commutes with $T$ and the conjugation by $1+Y$ does not modify the framed deformation class. The hypothesis 2 on the mutual endomorphisms of the $\rho_i$ implies that $Y$ must be a block diagonal matrix of the form
\begin{eqnarray}
Y=diag[Y_1,\dots,Y_r]
\end{eqnarray}
where $Y_i=A_i\otimes{id_2}\in{M_{2e_i}(R)}$ and $A_i\in{M_{e_i}(R)}$.

Now we need to choose properly the entries $\{a_{ist}\}_{s,t=1,\dots,e_i}$ of the matrices $Y_i$. Let
\begin{eqnarray}
(1+M)(1+Y)T(1+Y)^{-1}(1+M)^{-1}=(1+\tilde{M})T(1+\tilde{M})^{-1}.
\end{eqnarray}
We set
\begin{eqnarray}
a_{ist}=\frac{-x_{h_i+2s,h_i+2t}}{1+x_{h_i+2s,h_i+2t}}.
\end{eqnarray}
The resulting matrix $\tilde{M}$ has entries $\tilde{x}_{u,v}$ given by
\begin{eqnarray}
\tilde{x}_{u,v}=\begin{cases} 0 & \text{if} \  u=h_i+2s,\ v=h_i+2t \\ \frac{x_{u,v}}{1+x_{h_i+2s,h_i+2t}} & \text{otherwise}  \end{cases}
\end{eqnarray} 

                               
To make the notation easier, we rename $\tilde{M}=M$ and $\tilde{x}_{i,j}=x_{i,j}$. We call $(\tilde{\rho},\tilde{\beta}(1+Y))$ the resulting framed deformation obtained at the end of this process.

The framed deformation $\tilde{\rho}$ has values in the ring
\begin{eqnarray}
\tilde{R}=W(k)[[x_{1,1},\dots,x_{2n,2n}]]/(x_{h_i+2s,h_i+2t}: s,t=1,\dots,e_i,\ i=1,\dots,r)
\end{eqnarray}

We need to show that this is effectively the universal framed deformation. Observe that $\tilde{R}$ is a power series ring over $W(k)$ in exactly $N$ variables. First we need to compute the dimension of the framed tangent space. We use the fact that the tangent space $F_{\bar{\rho},S}^{\Box}(k[\epsilon])$ fits the exact sequence
\begin{eqnarray}
0\rightarrow{F_{\bar{\rho},S}(k[\epsilon])}\rightarrow{F_{\bar{\rho},S}^{\Box}(k[\epsilon])}\rightarrow{Ad(\bar{\rho})/Ad(\bar{\rho})^G}\rightarrow{0}
\end{eqnarray}
and that the unframed tangent space is trivial, because of the triviality of the extension set. Note that
\begin{multline}
Ad(\bar{\rho})^G=End_G(\bar{\rho}_1^{e_1}\oplus\dots\oplus{\bar{\rho}_r^{e_r}})=\mathop{\oplus}_{i=1}^r{End_G(\bar{\rho}_i^{e_i})}=\mathop{\oplus}_{i=1}^r{M_{e_i}(k)}
\end{multline}
where we have used the hypothesis on the sets $Hom_G(V_{\bar{\rho}_i},V_{\bar{\rho}_j})$. 

Therefore we have
\begin{eqnarray}
dim(F_{\bar{\rho},S}^{\Box}(k[\epsilon]))=dim(Ad(\bar{\rho}))-dim(Ad(\bar{\rho})^G)=4n^2-\sum_{i=1}^r{e_i}^2=N,
\end{eqnarray}
then the universal framed deformation ring $R_{\bar{\rho},S}^{\Box}$ and $\tilde{R}$ have the same relative Krull dimension.

Now we use the universality of $R_{\bar{\rho},S}^{\Box}$ that gives us a unique $W(k)$-algebra morphism $\pi:R_{\bar{\rho},S}^{\Box}\rightarrow{\tilde{R}}$ such that $\hat{\pi}\circ{\rho_{univ}}=\tilde{\rho}$ where $\rho_{univ}$ is the universal representation and $\hat{\pi}$ is the extension of $\pi$ to $GL_2$. We have a diagram
\[ \xymatrix {
W(k)[[x_1,\dots,x_N]] \ar[rd]^{\pi_2} \ar[r]^{\pi_1} & R_{\bar{\rho},S}^{\Box} \ar[d]^{\pi} & \\
 &  \tilde{R} & \simeq{W(k)[[y_1,\dots,y_N]]}. } \]
If the map $\pi$ is surjective, since $\pi_1$ is surjective, too, it follows that $\pi_2$ is surjective, too. But $\pi_2$ is $W(k)$-algebra map between algebras of the same dimension and therefore it must be an isomorphism. But then $\pi$ must be an isomorphism, too. The theorem is therefore proved, provided that $\pi$ is surjective.

To prove that the map $\pi$ is surjective, it is enough to show that the induced map on mod $p$ tangent space
\begin{eqnarray}
\tilde{\pi}:Hom(\tilde{R}/p,k[\epsilon])\rightarrow{Hom(R_{\bar{\rho},S}^{\Box}/p,k[\epsilon])}
\end{eqnarray}
is injective (because the functor $Hom(.,k[\epsilon])$ is contravariant). Since $\tilde{R}$ is a power series ring over $W(k)$ in $N$ variables, an element of $Hom(\tilde{R}/p,k[\epsilon])$ is given by a map which sends the variables $x_1,\dots,x_N$ to elements $\epsilon\alpha_1,\dots,\epsilon\alpha_N$ with $\alpha_1,\dots,\alpha_N$ giving a basis for the $k[\epsilon]$-module $V$ given by a representation lifting $V_{\bar{\rho}}$; different elements are given by different choices of the basis.
Suppose then that two elements $(V,\{\alpha_i\}),(V,\{\alpha_j\})$ have the same image with respect to $\tilde{\pi}$; it means that there exists a matrix $A\in{GL_n(k[\epsilon])}$ whose conjugation maps the basis $\{\alpha_i\}$ into $\{\alpha_j\}$ and $A$ commutes with the representation, that is, lies in the centralizer of the image of $\bar{\rho}$. But, beacuse of the construction of the representation $\bar{\rho}$ such matrix must be the identity. Therefore the map is injective and the theorem is proved.

\end{proof}
\vskip 0.5cm

REMARK: The existence of a $p$-adic lift to $W(k)$ for each $\bar{\rho}_i$ can be proved even if do not have the triviality of the extension set, in the case $p$ odd, if we had the condition that 
\begin{eqnarray}
\bar{\rho}_i|_{G_{\mathbb{Q}(\sqrt{\pm{p}})}}.
\end{eqnarray}
This condition is automathically implied by the absolute irreducibility of $\bar{\rho}_i$ if $p>3$ (see \cite[Ch.4.2]{K1} for a proof). The case $p=3$ si treated in \cite[Lemma 3.1 and 3.2]{D}.

\vskip 0.5cm
As an application of the theorem, we recover an example of Schoof in \cite{S2}. Consider an abelian variety $A$ over $\mathbb{Q}$ which has good reduction in all but one prime $\ell$, where it has semistable reduction. By \cite[Th. 1.2]{S2}, if $\ell=11$ then $A$ is isogenous to a product of copies of $E=J_0(11)$. Moreover $A$ is supersingular at $2$ and $A[2]\simeq{E[2]^g}$. Therefore we can look at the natural $G_{\mathbb{Q}}$-representation $\bar{\rho}_{A,2}$ on the $2$-torsion points of $A$ as product of $g$ copies of the reperesentation $\bar{\rho}_{E,2}$. In formulas
\begin{eqnarray}
\bar{\rho}_{A,2}=\mathop{\oplus}_{i=1}^g{\bar{\rho}_{E,2}}.
\end{eqnarray}
Then we can study the deformations of $\bar{\rho}_{A,2}$ from the ones of $\bar{\rho}_{E,2}$. Applying the theorem we have that the functor $F_{\bar{\rho}_{A,2},\underline{D}}$ is represented by a power series ring over $\mathbb{Z}_p$ in $3g^2$ variables. Moreover, if we go trough the same construction as in the theorem, we have that the universal framed deformation is given taking the product of $g$ copies of the $\mathbb{Z}_p$-representation given by the Tate module $T_p{E}$ and then applying the transformation with the matrix $M$.

\end{document}